\documentclass[12 pt]{amsart}

\usepackage{amscd,amsfonts,amssymb,amsmath}
\usepackage{tikz-cd}
\usepackage{xcolor}
\usepackage{hyperref}

\newtheorem{theorem}{Theorem}[section]
\theoremstyle{definition}

\newtheorem{corollary}[theorem]{Corollary}
\newtheorem{lemma}[theorem]{Lemma}
\newtheorem{proposition}[theorem]{Proposition}
\newtheorem{question}[theorem]{Question}

\newtheorem{definition}[theorem]{Definition}
\newtheorem{example}[theorem]{Example}
\newtheorem{remark}[theorem]{Remark}
\numberwithin{equation}{section}

\newcommand{\Aut}{\operatorname{Aut}}
\newcommand{\End}{\operatorname{End}}

\newcommand{\im}{\operatorname{Im}}

\newcommand{\Inn}{\operatorname{Inn}}

\newcommand{\id}{\mathrm{id}}

\title{Relative Rota--Baxter operators on groups and \\ Hopf algebras
}
\author[Valeriy G.~Bardakov,  Igor M.~Nikonov]{Valeriy G.~Bardakov,  Igor M.~Nikonov}

\date{\today}

\begin{document}
\sloppy



\maketitle
\begin{abstract}
 M.~Goncharov  introduced and studied a Rota--Baxter operator on a cocommutative Hopf algebra.
In the present paper  we define relative  Rota--Baxter operators on an arbitrary Hopf algebra. A particular case of this definition is Goncharov's operator.
On a Hopf algebra with a relative Rota--Baxter operator we define new associative operation and construct a new Hopf algebra and Hopf brace.
Further, we construct Rota--Baxter operators  of integer weights on some groups. The question on a possibility to define operator of zero weight on groups was formulated by X. Gao, L. Guo, Y. Liu, and Z.-C. Zhu. 
In the last section we construct a family of two generated Hopf algebras.  This family includes some known Hopf algebras, in particular,  4-dimensional Sweedler algebra $H_4$.


\end{abstract}

\maketitle
\tableofcontents

\section{Introduction}

Rota--Baxter groups were introduced in 2020 by L.~Guo, H.~Lang, and Y.~Sheng~\cite{Guo2020}.
A~\emph{Rota--Baxter group} is a group~$G$ endowed with a~map $B\colon G\to G$
satisfying the identity
$$
B(g)B(h) = B( g B(g) h B(g)^{-1} ),
$$
where $g,h\in G$.
This notion appeared as a group analogue of Rota--Baxter operators
defined on an algebra.
Rota--Baxter operators on algebras are known since the middle of the previous
century~\cite{Baxter,Tricomi} and they have in turn connections
with mathematical physics (classical and quantum Yang--Baxter equations),
number theory, operad theory, Hopf algebras,
combinatorics et cetera, see the monograph~\cite{GuoMonograph}.
After the initial work~\cite{Guo2020},
the study of Rota--Baxter groups have been continued in~\cite{BG2, BG1, Goncharov2020,JSZ}.

Further relative Rota--Baxter operators on groups were defined~\cite{JSZ}, which generalize Rota--Baxter operators on groups, and their connections with relative Rota--Baxter operators on algebras were found.
Let $G$ and $H$ be groups, and $\Psi\colon G\rightarrow \Aut(H)$ an action of $G$ on $H.$ A~map $B\colon H\rightarrow G$ is called a \emph{relative Rota-Baxter operator} on $H$ with respect to the action $\Psi$ if
$$
B(h)B(k)=B\bigl(h\, \Psi_{B(h)}(k)\bigr),~~~h, k \in H.
$$
In this case the quadruple $(H, G, \Psi, B)$ is said to be a \emph{relative Rota-Baxter group} (RRB-group).

If $H = G$ and
$$
\Psi \colon G \to \Inn(G), ~~\Psi_g = \Psi(g) \colon x \mapsto g x g^{-1},~~ x \in G,
$$
then the RRB-operator satisfies the identity
$$
B(h)B(k)= B\bigl(h\, \Psi_{B(h)}(k)\bigr) = B\bigl(h \, B(h) \, k \, B(h)^{-1}\bigr),  ~~h, k \in G,$$
and is the usual RB-operator on $G$ and $(G, B)$ is a RB-group.

In \cite{JSZ} was proved that if
 $H$ and $G$ are groups and $\Psi \colon G \to \Aut(H)$ is an action of $G$ on $H$. Then a map $B \colon H \to G$
 is a RRB-operator if and only if the set
$$
Gr(B) = \{ (B(h), h)~|~h \in H \}
$$
is a subgroup of the semi-direct product $H \rtimes_{\Psi} G$.
The map $B \colon (H, \circ_B) \to G$  is a group homomorphism and $(H, \circ_B) \cong Gr(B)$ as groups via the the map
$h \mapsto (B(h), h)$.
In \cite{BG2} was found connection between  Rota--Baxter groups and skew left braces. Similar result for relative Rota--Baxter groups was found in \cite{BGST, RS-23}.

In~\S~\ref{Prel}, we state the required preliminaries on Hopf algebras, Rota--Baxter operators on algebras  and on  groups, relative Rota--Baxter operators on  groups.

 M.~Goncharov \cite{Goncharov2020} introduced and studied a Rota--Baxter operator on a cocommutative Hopf algebra
that generalizes the notions of a Rota--Baxter operator on a group and a Rota--Baxter operator of weight 1 on a Lie algebra.
In the present paper  we define relative  Rota--Baxter operators and Rota---Baxter operators on an arbitrary Hopf algebra. A particular case of this definition is Goncharov's operator. On a Hopf algebra with a relative Rota--Baxter operator we define a new binary operation  and find conditions under which this operation is associative.
Using this operation we construct a new Hopf algebra and a Hopf brace (see Section  \ref{RRBO}).

On arbitrary group it is possible to define Rota--Baxter operators of weight -1 and 1. In~\S~\ref{weight} we construct Rota--Baxter operators of integer weights on some groups. The question on a possibility to define operators of zero weight on groups was formulated in \cite[p. 4]{GGLZ}.

In the last section we construct a family of two generated Hopf algebras.  This family includes some known Hopf algebras, in particular,  4-dimensional Sweedler algebra $H_4$. We find automorphism group for some of these algebras.

\bigskip


\section{Preliminaries} \label{Prel}

\subsection{Hopf algebras} Recall some facts from Hopf theory (see, for example, \cite[Chapter~11]{T}).

 Let $A$ be an associative algebra with unit $1_A$ over a commutative ring with unit, $\Bbbk$. Assume that $A$ is provided with multiplicative
$\Bbbk$-linear homomorphisms
$$
\Delta \colon A \to A^{\otimes 2} = A \otimes_{\Bbbk} A~\mbox{and}~\varepsilon \colon A \to \Bbbk,
$$
called
the comultiplication and the counit respectively, and a $\Bbbk$-linear map $S \colon A \to A$, called the antipode. It is understood that $\Delta(1_A) = 1_A \otimes 1_A$ and $\varepsilon(1_A) = 1$. The tuple $(A, \Delta, \varepsilon, S)$ is said to be a \emph{Hopf algebra} if these
homomorphisms satisfy together with the multiplication $m \colon A \times A \to A$ the following identities:
\begin{equation}\label{1.1.a}
(\id_A  \otimes \Delta) \Delta  = (\Delta \otimes \id_A) \Delta,
\end{equation}
\begin{equation}\label{1.1.b}
m (S \otimes \id_A ) \Delta  = m (\id_A \otimes S) \Delta = \varepsilon \cdot 1_A,
\end{equation}
\begin{equation}\label{1.1.c}
(\varepsilon \otimes \id_A ) \Delta  = (\id_A \otimes  \varepsilon) \Delta = \id_A.
\end{equation}
Note that to  write down the first equality we identify $(A \otimes A) \otimes A = A \otimes (A \otimes A)$ via $(a \otimes b) \otimes c = a \otimes (b \otimes c)$, where $a, b, c \in A$.  Similarly, to write down the third equality, we identify $\Bbbk \otimes A = A \otimes \Bbbk = A$ via $1 \otimes a = a \otimes 1 = a$. The axioms imply that the antipode $S$ is an antiautomorphism of both the
algebra and the coalgebra structures in~$A$. This means that
$$
m (S \otimes S) = S \circ m \circ P_A \colon A^{\otimes 2} \to A,~~~~P_A (S \otimes S) \Delta = \Delta  \circ S  \colon A \to A^{\otimes 2},
$$
where $P_A$ denotes the flip
$$
a \otimes b \mapsto b \otimes a \colon A^{\otimes 2} \to A^{\otimes 2}.
$$
 It also follows from the axioms that $S(1_A) = 1_A$ and $\varepsilon \circ S = \varepsilon \colon A \to \Bbbk$.

\begin{example}
\begin{enumerate}
\item A group ring $\Bbbk[G]$ is a Hopf algebra if the homomorphisms $\Delta, \varepsilon, S$  are defined on the additive generators $g \in  G$ by the formulas
$$
\Delta(g) = g \otimes g,~~ \varepsilon(g) = 1,~~ S(g) = g^{-1}.
$$

\item The universal enveloping algebra $U(\mathfrak{g})$ of a Lie algebra $\mathfrak{g}$ is a Hopf algebra, if the homomorphisms $\Delta, \varepsilon, S$  are defined on the multiplicative  generators $g \in \mathfrak{g}$ by the formulas
$$
\Delta(g) = g \otimes 1 + 1 \otimes g,~~ \varepsilon(g) = 0,~~ S(g) = -g.
$$
\item The famous non-cocommutative 4-dimension Sweedler  algebra $H_4$ as algebra is generated by two elements $x$, $g$ with multiplication
$$
g^2 = 1,~~x^2 = 0,~~x g = -g x.
$$
Comultiplication, counit and antipod are defined by the rules
$$
\Delta(g) = g \otimes g,~~\Delta(x) = x \otimes 1 + g \otimes x,~~ \varepsilon(g) = 1,~~ \varepsilon(x) = 0,~~ S(g) = g^{-1} = g,~~S(x) = -g x.
$$
The antipode $S$ has order 4 and for any $a \in H_4$ we have $S^2(a) = g a g^{-1}$.
\end{enumerate}
\end{example}

 Throughout the paper  we shall use the
Sweedler's notation with superscripts for all the comultiplications,
$$
\Delta(h) = h_{(1)} \otimes h_{(2)},~~~h \in H.
$$

\subsection{Rota--Baxter operators}
Rota--Baxter operators for commutative algebras first appeared in the paper of G.~Baxter~\cite{Baxter}.
For basic results and the main properties of Rota--Baxter algebras see~\cite{GuoMonograph}.

\begin{definition}
Let $A$ be an algebra over a field~$\Bbbk$. A linear operator $R$ on $A$ is called
a \emph{Rota--Baxter operator of weight~$\lambda\in \Bbbk$} if
\begin{equation}\label{RBAlgebra}
R(x)R(y) = R( R(x)y + xR(y) + \lambda xy ).
\end{equation}
for all $x,y\in A$.
An algebra endowed with a Rota--Baxter operator is called a~\emph{Rota--Baxter algebra}.
\end{definition}

In \cite{Guo2020}  the Rota--Baxter operator on  groups was defined.

\begin{definition}
Let $G$ be a group.
\begin{enumerate}
\item A map $B \colon G\to G$ is called a {\it Rota--Baxter operator of weight~1} if
\begin{equation}\label{RB}
B(g)B(h) = B( g B(g) h B(g)^{-1} )
\end{equation}
for all $g,h\in G$.
\item A map $C \colon G\to G$ is called a {\it Rota--Baxter operator of weight~$-1$} if
$$
C(g) C(h) = C( C(g) h C(g)^{-1} g )
$$
for all $g,h\in G$.

\end{enumerate}
\end{definition}

There is a~bijection between Rota--Baxter operators of weights~1 and $-1$ on a~group~$G$. We will call Rota--Baxter
operators of weight~1 simply Rota--Baxter operators (RB-operators). A group endowed with a Rota--Baxter operator is called
a~{\it Rota--Baxter group} (RB-group).

In the  paper \cite{Guo2020}
it was proved, that if $(G, B)$ is a Rota--Baxter Lie group, then the tangent map of $B$
at the identity is a Rota--Baxter operator of weight 1 on the Lie algebra of the Lie group~$G$.

The next lemmas easy follows from the definition.

\begin{lemma}\label{lem:elementary}
If $B\colon G \to G$ is a Rota--Baxter operator on a group $G$, then

a) $B(e) = e$;

b) $B(g)B(g^{-1}) = B([g^{-1},B(g)^{-1}])$;

c) $B(g)B(B(g)) =  B(g B(g))$;

d) If $B(g) = e$ for some non-trivial $g\in G$, then $B(h) = B(gh)$ for any $h\in G$,

e) $B(g)^{-1} = B(B(g)^{-1}g^{-1}B(g))$ for any $g \in G$.
\end{lemma}

\begin{lemma}\label{KerImSub}
Let $B$ be a Rota--Baxter operator on a group~$G$. Then

a) $\ker B$ is a subgroup of $G$,

b) $\im B$ is a subgroup of $G$.
\end{lemma}

The next proposition shows that a  Rota--Baxter group possesses another group multiplication.

\begin{proposition}[\cite{Guo2020}]\label{prop:Derived}
Let $(G, \cdot ,B)$ be a Rota--Baxter group.

a) The pair $(G, *)$, with the multiplication
\begin{equation}\label{R-product}
g*h = gB(g)hB(g)^{-1},
\end{equation}
where $g,h\in G$, is also a group.

b) The operator $B$ is a Rota--Baxter operator on the group $(G, *)$.

c) The map $B\colon(G, *)\to (G, \cdot)$ is a homomorphism of Rota--Baxter groups.
\end{proposition}


\subsection{Relative Rota-Baxter operators}

The next definition was introduced in ~\cite{JSZ}.

\begin{definition}[{\cite{JSZ}}]
A \emph{relative Rota--Baxter group} is a quadruple $(H,G,\Phi, B)$, where $H$ and $G$
are groups, $\Phi\colon G\to \Aut(H)$ a group homomorphism and $B\colon H\to G$ is a map satisfying the condition
\begin{equation}\label{eq:relative_group_RB}
B(h_1)B(h_2)=B(h_1\Phi_{B(h_1)}(h_2))
\end{equation}
for all $h_1, h_2\in H$.
The map $B$ is referred as the \emph{relative Rota--Baxter operator} on $H$.
\end{definition}

\begin{example}
If $H=G$ and $\phi=Ad$, i.e. $\phi_g(h)=ghg^{-1}$ for $g, h \in H$, we get the usual Rota--Baxter group.
\end{example}

The next definition is analogous to the previous one.

\begin{definition}[{\cite{JSZ}}]
Let $\phi\colon\mathfrak g\to Der(\mathfrak h)$ be an action of a Lie algebra $(\mathfrak g, [\cdot,\cdot]_{\mathfrak g})$ on a Lie algebra
$(\mathfrak h, [\cdot,\cdot]_{\mathfrak h})$. A linear map $B\colon\mathfrak h\to\mathfrak g$ is called a \emph{relative Rota--Baxter operator of weight $\lambda \in \Bbbk$} on $\mathfrak g$
with respect to $(\mathfrak h; \phi)$ if
\begin{equation}\label{eq:relative_Lie_RB}
[B(u),B(v)]_{\mathfrak g} = B\left( \phi(B(u))v-\phi(B(v))u+\lambda[u,v]_{\mathfrak h} \right), \quad \mbox{for~all}~ u,v\in\mathfrak h.
\end{equation}

\end{definition}

\begin{example}\label{ex:lambda_lie_RB}
Let $\mathfrak h=\mathfrak g$, $[\cdot,\cdot]_{\mathfrak h}=\lambda[\cdot,\cdot]_{\mathfrak g}$ and $\phi(u)(v)=[u,v]_{\mathfrak g}$, $u,v\in\mathfrak g$. Then the relative Rota--Baxter relation (of weight $1$) looks like
\begin{equation}\label{eq:lambda_lie_RB}
[B(u),B(v)]_{\mathfrak g}=B([B(u),v]_{\mathfrak g}-[B(v),u]_\mathfrak g+\lambda[u,v]_{\mathfrak g})
\end{equation}
and coincide with the definition of Rota--Baxter operator of weight $\lambda$ on $\mathfrak g$.
\end{example}

A connection between relative Rota--Baxter operators on Lie groups and relative Rota--Baxter operators on the corresponding Lie algebras is described by the following theorem.

\begin{theorem}[{\cite[Theorem 4.1]{JSZ}}]\label{thm:differetial_relative_RB}
Let $G$, $H$ be Lie groups and $\mathfrak g$, $\mathfrak h$ their Lie algebras.
Let $\mathcal B\colon H \to G$ be a relative Rota--Baxter operator on $G$ with respect to an action $(H, \Phi)$. Define $B\colon \mathfrak h\to\mathfrak g$ by
\[
B = {\mathcal B}_{*1_H}
\]
which is the tangent map of $\mathcal B$ at the identity $1_H$. Then $B$ is a relative Rota--Baxter operator of weight $1$ on $\mathfrak g$ with respect to the action $(\mathfrak h, \phi)$, where $\phi\colon\mathfrak g\to Der(\mathfrak h)$ is the differentiated action of $\Phi$.
\end{theorem}

\bigskip


\section{Relative Rota-Baxter operators on Hopf algebras} \label{RRBO}

In the present section we define a (relative) Rota--Baxter operators on a Hopf algebra~$H$.
If $H$ is cocommutative, then M.~Goncharov \cite{Goncharov2020} defined such operators $B \colon H \to H$, which is coalgebra homomorphism and satisfies the identity
$$
B(x) B(y) = B\left( x_{(1)} \, B(x_{(2)}) \, y \, S(B(x_{(3)})) \right),~~x, y \in H,
$$
where
$$
\Delta(x) = x_{(1)} \otimes x_{(2)}, ~~\Delta^2(x) = (\id \otimes \Delta ) \Delta(x) = x_{(1)} \otimes x_{(2)} \otimes x_{(3)}.
$$
In particular, if $x$ is a group-like element, then $x = x_{(1)}  = x_{(2)}= x_{(3)}$, and $S(x) = x^{-1}$. In this case we have
$$
B(x) B(y) = B\left( x \, B(x) \, y \, B(x)^{-1} \right),
$$
that is the  definition of the Rota--Baxter operator on groups.

\subsection{Relative Rota-Baxter operators}

We introduce the next definition.

\begin{definition} \label{Def-RRBO}
Let $(H, *, \Delta_H, \varepsilon_H, S_H)$,  $(G, \cdot, \Delta_G, \varepsilon_G, S_G)$ be  Hopf algebras and the next conditions hold:

1) A map
$$
B \colon H(\Delta_H, \varepsilon_H) \to G(\Delta_G, \varepsilon_G)
$$
is a coalgebra homomorphism.

2) $\Phi \colon G \to \End(H)$ defines an action of $G$ on $H$, i.e. $H$ is a left $G$-module algebra. It means that the next equalities hold,
$$
\Phi_g \Phi_h (a)= \Phi_{gh}(a),~~\Phi_1 (a)= a,~~\Phi_g (a * b)= \Phi_{g_{(1)}} (a) * \Phi_{g_{(2)}} (b),~~\Phi_g (1)=  \varepsilon(g) 1.
$$

3) For any $a, b \in H$ holds
\begin{equation} \label{Cond-as}
\left( \Phi_{B(a_{(2)})} (b) \right)_{(1)} \otimes a_{(1)} * \left(  \Phi_{B(a_{(2)})} (b) \right)_{(2)} = \Phi_{B(a_{(1)})} (b_{(1)})  \otimes a_{(2)} *  \Phi_{B(a_{(3)})} (b_{(2)}).
\end{equation}

4) For any $a, b \in H$ holds
$$
B(a) B(b) = B\left(a_{(1)} * \Phi_{B(a_{(2)})} (b)\right).
$$
We shall call such  operator $B$ a {\it relative Rota-Baxter operator} (RRBO) on the Hopf algebra $H$ with respect to $\Phi$.
 The  quadruple $(H, G,\Phi, B)$ is called by a \emph{relative Rota--Baxter Hopf algebra}.
\end{definition}

If $B$ is a RRBO, then we can define a binary operation $\circ \colon H \otimes H \to H$ by the rule
$$
a \circ b =  a_{(1)} * \Phi_{B(a_{(2)})} (b).
$$
Hence
$$
B(a) B(b) = B(a \circ b).
$$

\begin{proposition}
A relative Rota-Baxter operator $B \colon H \to G$ sends the unit element $1 \in H$ to the unit element of $G$.
\end{proposition}

\begin{proof}
Let us put $a = b = 1$ in the equality
$$
B(a) B(b) = B\left(a_{(1)} * \Phi_{B(a_{(2)})} (b)\right),
$$
we get in $G$ that $B(1) B(1) = B(1)$. It means that $B(1)$ is a unit element in $G$.
\end{proof}

From this proposition and from the fact that $B$ is a coalgebra homomorphism follows that $B$ sends group-like elements to the group-like elements and primitive elements to primitive elements.

\begin{remark}
It is easy to see that the equality (\ref{Cond-as}) is equivalent to the equality
\begin{equation*}
\left( \Phi_{B(a)} (b) \right)_{(1)} \otimes  \left(  \Phi_{B(a)} (b) \right)_{(2)} = \Phi_{B(a_{(2)})} (b_{(1)})  \otimes S_H(a_{(1)}) * a_{(3)} *  \Phi_{B(a_{(4)})} (b_{(2)}).
\end{equation*}
\end{remark}

\begin{example}
\begin{enumerate}
\item Let $G = H$, if we define $\Phi$ by the formula
$$
 \Phi_a (b) = a_{(1)} \, b \, S(a_{(2)}),
$$
then
$$
a \circ b = a_{(1)} * B(a_{(2)}) \, b \,  S(B(a_{(3)})),
$$
and  $B \colon H \to H$  satisfies   the equality
$$
B(a) \, B(b) = B\left(a_{(1)} * B(a_{(2)}) \, b \,  S(B(a_{(3)}))\right).
$$
This operator was defined in \cite{Goncharov2020}.  We shall call this operator by the {\it group Rota--Baxter operator} (GRBO).

The condition on associativity of $\circ$ has the form
$$
B(a_{(2)}) \, b_{(1)} \, S \left( (B(a_{(5)}))_{(1)} \right) \otimes a_{(1)} \left( B(a_{(3)})  \, b_{(2)} \, S(B(a_{(4)})) \right)_{(2)} =
$$
$$
= B(a_{(1)})  \, b_{(1)}\,  S(B(a_{(2)})) \otimes a_{(3)} B(a_{(4)})  \, b_{(2)} \, S(B(a_{(5)})).
$$
If $H$ is a cocommutative algebra, then this condition holds.


\item Let a Hopf algebra $H$ have an exact factorization, i.e. there exist Hopf subalgebras $A$, $L$ in $H$ such that $H=AL$ and $A\cap L=\Bbbk 1$. We take $G= L^{op}$ and the map
$\Phi \colon G \to \End(H)$ defined  by the formula
$$
\Phi_{l}(h)=S_H(l_{(1)})h l_{(2)}, \quad h\in H, l\in L^{op},
$$
and the map $B\colon H\to L^{op}$,
$$
B(al)=\varepsilon(a)l,\quad a\in A,l\in L.
$$
Then $B$ is a relative Rota--Baxter operator on $H$ with respect to $\Phi$. This example generalises the example \cite[Proposition 2]{Goncharov2020} to the case of non-cocomutative Hopf algebras.
\end{enumerate}
\end{example}

\begin{proposition}
The algebra $(H, \circ)$ is an associative algebra.
\end{proposition}

\begin{proof}
We have to check the equality
$$
(a \circ b) \circ c = a \circ (b \circ c),~~a, b, c \in H.
$$
The left hand side has the form
$$
(a \circ b) \circ c = (a \circ b)_{(1)} *  \Phi_{B((a \circ b)_{(2)})} (c) = \left( a_{(1)} *  \Phi_{B(a_{(2)})} (b) \right)_{(1)}
 *  \Phi_{B((  a_{(1)} *   \Phi_{B(a_{(2)})}(b) )_{(2)})} (c) =
 $$
 $$
 =  a_{(1)} *  \left( \Phi_{B(a_{(3)})} (b) \right)_{(1)}  *  \Phi_{B(  a_{(2)} * (\Phi_{B(a_{(3)})}(b) )_{(2)})} (c),
$$
where we used the equality
$$
 (a * b)_{(1)} \otimes  (a * b)_{(2)} =  \left( a_{(1)} * b_{(1)}  \right) \otimes \left( a_{(2)} * b_{(2)} \right).
$$

The right hand side has the form
$$
a \circ (b \circ c) = a_{(1)} *  \Phi_{B(a_{(2)})} \left( b_{(1)} *  \Phi_{B(b_{(2)})} (c) \right) =
a_{(1)} *  \Phi_{(B(a_{(2)}))_{(1)}} (b_{(1)}) * \Phi_{(B(a_{(2)}))_{(2)}} \left( \Phi_{B(b_{(2)})} (c) \right) =
 $$
 $$
= a_{(1)} *  \Phi_{(B(a_{(2)}))_{(1)}} (b_{(1)}) * \Phi_{(B(a_{(2)}))_{(2)} B(b_{(2)})} (c) =
a_{(1)} *  \Phi_{B(a_{(2)})} (b_{(1)}) * \Phi_{B(a_{(3)} * \Phi_{B(a_{(4)})}(b_{(2)}))} (c).
$$
Comparing the left hand side and the right hand side, we get
$$
 \left( \Phi_{B(a_{(3)})} (b) \right)_{(1)}  *  \Phi_{B(  a_{(2)} * (\Phi_{B(a_{(3)})}(b) )_{(2)})} (c) =   \Phi_{B(a_{(2)})} (b_{(1)}) * \Phi_{B(a_{(3)} * \Phi_{B(a_{(4)})}(b_{(2)}))} (c).
$$
This equality follows from (\ref{Cond-as}) for $a = a_{(2)}$ and by  applying the bilinear map $x \otimes y \mapsto x * \Phi_{B(y)}$.
\end{proof}

\begin{theorem}
The algebra $(H, \circ, \Delta_H, \varepsilon_H, S_B)$, where
$$
S_B (a) = \Phi_{S_G(B(a_{(1)}))} (S_H(a_{(2)})),~~a \in H,
$$
 is a Hopf algebra.
\end{theorem}

\begin{proof}
For simplicity we will write $\Delta = \Delta_H$, $\varepsilon = \varepsilon_H$.

Let us check that the comultiplication is an algebra homomorphism. We have
$$
\Delta (a \circ b) = \Delta\left( a_{(1)} * \Phi_{B(a_{(2)})}(b)\right) =  a_{(1)}  * \left( \Phi_{B(a_{(3)})}(b)\right)_{(1)} \otimes
 a_{(2)}  * \left( \Phi_{B(a_{(3)})}(b)\right)_{(2)}.
$$
On the other side,
$$
\Delta (a) \circ  \Delta (b) = a_{(1)} \circ b_{(1)} \otimes   a_{(2)} \circ b_{(2)} =  a_{(1)} * \Phi_{B(a_{(2)})}(b_{(1)}) \otimes  a_{(3)} * \Phi_{B(a_{(4)})}(b_{(2)}).
$$
Hence, we equality $\Delta (a \circ b) = \Delta (a) \circ  \Delta (b)$ holds if and only if,
$$\left( \Phi_{B(a_{(3)})}(b)\right)_{(1)} \otimes
 a_{(2)}  * \left( \Phi_{B(a_{(3)})}(b)\right)_{(2)} = \Phi_{B(a_{(2)})}(b_{(1)}) \otimes  a_{(3)} * \Phi_{B(a_{(4)})}(b_{(2)}).
$$
 This equality follows from (\ref{Cond-as}) for $a = a_{(2)}$.

Now we have to check
$$
a_{(1)} \circ S_B(a_{(2)}) = \varepsilon(a) 1.
$$
By the definition of $\circ$ we have
\begin{multline*}
   a_{(1)} \circ S_B(a_{(2)})=a_{(1)}*\Phi_{B(a_{(2)})}\Phi_{S_G(B(a_{(3)})}(S_H(a_{(4)}))=\\
   a_{(1)}*\Phi_{B(a_{(2)})_{(1)}S_G(B(a_{(2)})_{(2)})}(S_H(a_{(3)}))=a_{(1)}*\Phi_{\varepsilon(a_{(2)})1}(S_H(a_{(3)}))=\\ a_{(1)}*S_H(a_{(2)})=\varepsilon(a) 1.
\end{multline*}

On the other hand, for any $a\in H$
$$
    a=\varepsilon(a_{(1)})a_{(2)}=a_{(1)}\circ S_B(a_{(2)})\circ a_{(3)}=a_{(1)}*\Phi_{B(a_{(2)})}(S_B(a_{(3)})\circ a_{(4)}).
$$
Hence,
$$
\Phi_{B(a_{(1)})}(S_B(a_{(2)})\circ a_{(3)})=S_H(a_{(1)})a_{(2)}=\varepsilon(a)1,
$$
and
\begin{multline*}
    S_B(a_{(1)})\circ a_{(2)}=\Phi_{S_G(B(a_{(1)}))}\Phi_{B(a_{(2)})}(S_B(a_{(3)})\circ a_{(4)})=\\
    \Phi_{S_G(B(a_{(1)}))}(\varepsilon(a_{(2)})1)=\varepsilon(a_{(1)})\varepsilon(a_{(2)})1=\varepsilon(a)1.
\end{multline*}
\end{proof}

\subsection{Hopf Rota--Baxter operators}

As particular case of relative Rota--Baxter operator on Hopf algebras we can define a Rota--Baxter operator on Hopf algebra. In the general construction we take
$G = H^{op}$

\begin{definition} \label{Def-RRBO}
Let $(H, *, \Delta_H, \varepsilon_H, S_H)$ be a Hopf algebra and the next conditions hold:

1) A map
$$
B \colon H \to H^{op}
$$
is a coalgebra homomorphism.

2) $\Phi \colon  H^{op} \to \End(H)$ is  defined by the rule  $\Phi_a(b) = S(a_{(1)}) \,  b (a_{(2)})$.

3) For any $a, b \in H$ holds
\begin{equation}
S\left( B(a_{(2)}) \right) \, b \, B(a_{(3)}) \otimes S\left( B(a_{(1)}) \right) = S\left(B(a_{(2)}) \right) \, b  \, B(a_{(3)}) \otimes S(a_{(1)}) \, a_{(4)} \, S\left( B(a_{(5)})\right).
\end{equation}

4) For any $a, b \in H$ holds
$$
B(a) B(b) = B\left(a_{(1)} *   S(B(a_{(2)})) \,  b B(a_{(3)})\right).
$$

We shall call such  operator $B$ a {\it Hopf Rota-Baxter operator} (HRBO) on $H$ with respect to $\Phi$, and the triple $(H, \Phi, B)$  a {\it Rota-Baxter Hopf algebra}.
\end{definition}

If we define a binary operation $\circ \colon H \otimes H \to H$ by the rule
$$
a \circ b = a_{(1)} *  S(B(a_{(2)})) \,  b \, B(a_{(3)}),
$$
then
$$
B(a) B(b) = B(a \circ b).
$$





\subsection{Hopf braces} Hopf brace was defined in the paper \cite{AGV}. The axiom which connects to associative operations has the form
\begin{equation} \label{Hop-br}
a \circ (b c) = (a_{(1)} \circ b) \, S(a_{(2)}) \, (a_{(3)} \circ c).
\end{equation}
If $a, b, c$ are group-like elements, then this axiom has the form
$$
a \circ (b c) = (a \circ b) \, a^{-1} \, (a \circ c).
$$
This is the axiom of skew brace.

In \cite{ZLMZ} was used other approach. Let $(H, \cdot, 1)$ be an algebra.  A Hopf brace structure over $H$ consists of the
following data:

1) a Hopf algebra structure $(H, \cdot, 1, \Delta, \varepsilon, S)$,

2) a Hopf algebra structure $(H, \cdot, 1, \Delta', \epsilon, T)$,

3) satisfying the following compatibility:
$$
a_{(1')} \otimes a_{(2'1)} \otimes a_{(2'2)}  = a_{(11')} \, S(a_{(2)}) \, a_{(31')} \otimes a_{(12')}  \otimes a_{(32')}
$$
for any $a \in  H$, where $\Delta(a)$ is denoted by $a_{(1)} \otimes a_{(2)}$ and $\Delta'(a)$ denoted by $a_{(1')} \otimes a_{(2')}$.

\begin{remark}
\begin{enumerate}
\item The condition 3) can be write in the form
$$
(\id  \otimes \Delta) \Delta'(a) = ( m \otimes \id \otimes \id) (m  \otimes \id \otimes \id \otimes \id) ( \id \otimes S \otimes \id \otimes \id \otimes \id)  P_{34} P_{23} (\Delta' \otimes \id  \otimes \Delta') ( \id \otimes \Delta) \Delta(a),
$$
where $m$ denotes the multiplication $\cdot$, $P_{ij}$ is the permutation of $i$-th and $j$-th components.

\item The condition 3) is dual to the condition (\ref{Hop-br}) and there exists a standard procedure to construct 3) from  (\ref{Hop-br}).

\item The definition   Hopf brace which formulated above is different from the definition in \cite{AGV}. May be, it is better to say on co-Hopf brace or Hopf cobrace.
\end{enumerate}
\end{remark}

 On a Hopf algebra $(H, *, \Delta_H, \varepsilon_H, S_H)$ we introduced the  Rota--Baxter operator and defined an associative  operation $\circ$.

\begin{proposition}
The new operation $\circ$ and the old antipode  $S_H$ define a Hopf brace  on~$H$.
\end{proposition}

\begin{proof}
We define the operation
$$
a \circ b = a_{(1)} * \Phi_{B(a_{(2)})} (b),
$$
where $*$ is the multiplication on $H$. We have to check (\ref{Hop-br}). We have
$$
a \circ (b * c)  = a_{(1)} * \Phi_{B(a_{(2)})} (b * c) = a_{(1)} * \Phi_{B(a_{(2)})} (b)  * \Phi_{B(a_{(3)})} (c) =
$$
$$
= a_{(1)} * \Phi_{B(a_{(2)})} (b) * S(a_{(3)}) * a_{(4)} * \Phi_{B(a_{(5)})} (c) = (a_{(1)} \circ b) * S(a_{(2)}) * (a_{(3)} \circ c), ~~a, b, c \in H.
$$
Here we used that $ \Phi_{B(a_{(2)})}$ is an automorphism and  $B$ is homomorphism of coalgebras.
\end{proof}

\medskip
The next questions are of interest.

\begin{question}

1) In \cite{BNY} were introduced and studied  $\lambda$-braces. Is it possible to define relative  $\lambda$-braces?

2) Find Hopf RBO on $U_q(sl_2)$.
\end{question}

\bigskip


\section{Rota--Baxter operators of weight $\lambda$ on groups} \label{weight}

In this section we are studying the next questions.
 Let $\lambda$ be an integer number. Is it possible to define on a group $G$ a Rota-Baxter operator of weight  $\lambda$? For $\lambda=0$ this question was formulated in
\cite[p. 4]{GGLZ}.

On arbitrary group it is possible to define RBO of weight -1 and 1 (see Section \ref{Prel}). There is bijection between these operators. Indeed (see~\cite[Remark 2.12]{Guo2020}), if $B$ is an RBO on a group of weight 1, put $C(a) = B(a^{-1})$ for any $a \in G$, then the definition of $B$ we can rewrite in the form
$$
C(a^{-1}) \, C(b^{-1}) = C\left((a \,  C(a^{-1}) \, b \, C(a^{-1})^{-1} )^{-1} \right).
$$
It is equivalent to the identity
$$
C(a^{-1}) \, C(b^{-1}) = C\left( C(a^{-1}) \, b^{-1} \, C(a^{-1})^{-1} a^{-1} \right),
$$
or
$$
C(a) \, C(b) = C\left( C(a) \, b \, C(a)^{-1} a \right).
$$

To formulate further  statements, recall a definition of skew brace (see, for example \cite{GV}).
An algebraic system $(G; \cdot, \circ)$ with two binary algebraic operations is called a  \emph{left skew brace}, if $(G; \cdot)$ and $(G; \circ)$ are groups, which are called additive and multiplicative, correspondingly, and
\begin{equation}
a \circ (b \cdot  c) =  (a \circ b) \cdot a^{-1} \cdot (a \circ c)
 \end{equation}
for all $a, b, c \in G$, where $a^{-1}$ denotes the additive inverse of $a$.

    Let $(G,\cdot)$ be a group. Let $(G,\ast)$ be another group operation which is compatible with conjugation $Ad$:
\[
    g\cdot(h_1\ast h_2)\cdot g^{-1}=(gh_1g^{-1})\ast(gh_2g^{-1}),\quad g,h_1,h_2\in G,
\]
    and such that the units of the group operations $\cdot$ and $\ast$ coincide $1^\cdot=1^\ast$.

    We will use the next evident observation

\begin{proposition}
Let $(G, \cdot)$ be a group and $f \colon G \to G$ a bijection, then the operation
$$
a * b = f^{-1} (f(a) f(b)),~~a, b \in G,
$$
is a group operation on $G$.
\end{proposition}


\medskip

Let us write down the relative Rota--Baxter relation for the case $H=(G, \ast)$, $\phi=Ad$
\begin{equation}\label{eq:ast_group_RB}
B(g_1)B(g_2)=B(g_1\ast B(g_1)g_2B(g_1)^{-1}),\quad g_1,g_2\in G.
\end{equation}

Consider the binary operation $g_1\circ g_2=g_1\ast B(g_1)g_2B(g_1)^{-1}$.

We have the following statement.

\begin{proposition}
   1) $(G, \circ)$ is a group;

   2) $(G, \ast,\circ)$ is a skew left brace;

   3) if $(G, \cdot, \ast)$ is a skew left brace then $(G, \cdot,\circ)$ is a skew left brace.
\end{proposition}
\begin{proof}
    The first two properties are particular cases of~\cite[Proposition 3.5]{JSZ} and~\cite[Proposition 3.5]{RS-23}. Let us prove the third statement
\begin{multline*}
    a\circ (bc)=a\ast a(bc)a^{-1}=a(a\ast (bc))a^{-1}=a((a\ast b)a^{-1}(a\ast c))a^{-1}=\\ a(a\ast b)a^{-1}a^{-1}a(a\ast c)a^{-1}=(a\ast aba^{-1})a^{-1}(a\ast aca^{-1})=(a\circ b)a^{-1}(a\circ c).
\end{multline*}
\end{proof}

Let us use relative Rota--Baxter operator on groups to construct Rota--Baxter operators of weight $\lambda$ on groups.

\begin{definition}
Let $\lambda\in\mathbb R\setminus\{0\}$. Let $G$ be a group such that the map $f\colon G\to G$, $f(g)=g^\lambda$, is a well-defined bijection. Define a new group operation on $G$ by the formula
\[
g\ast h= f^{-1}(f(g)f(h))=(g^\lambda h^\lambda)^{\frac 1\lambda}.
\]
We say that a map $B\colon G\to G$ is a \emph{Rota--Baxter operator of weight $\lambda$} if $B$ satisfy the equation~\eqref{eq:ast_group_RB} for the operation $\ast$ defined above, i.e. the equation
$$
B(g) B(h) = B\left(\left( g^{\lambda} \, B(g) \, h^{\lambda} B(g)^{-1}  \right)^{\frac 1\lambda} \right).
$$
\end{definition}

\begin{example}
1) If   $G$ is a simple-connected solvable Lie group, then the map $f\colon G\to G$, $f(g)=g^\lambda$, is a well-defined bijection.

2) If  $\lambda = -1$, then $a * b = b a$ and the RB-operator  of weight $-1$  satisfies the well-known identity,
$$
C(g) C(h) = C(C(g) h C(g)^{-1} g).
$$
\end{example}

As a consequence of Theorem~\ref{thm:differetial_relative_RB} we have

\begin{proposition}
    Let $\mathcal B$ be a Rota--Baxter operator of weight $\lambda \not= 0$ on a Lie group $G$ and $B={\mathcal B}_{*1_G}$ its differential at the unit of $G$. Then $B$ is a Rota--Baxter operator of weight $\lambda$ on the Lie algebra $\mathfrak g$ of the Lie group $G$.
\end{proposition}

\begin{proof}
    Let $H$ be the group Lie $(G,\ast)$. The commutator of its Lie algebra $\mathfrak h$ is a pull back of the commutator on $\mathfrak g$ by the differential $f_{*1_G}$ of the map $f$. Since for any $u\in\mathfrak g$ and $t\in\mathbb R$ $f(e^{tu})=e^{\lambda t u}$, we have $f_{*1_G}(u)=\lambda u$. Hence, $[\cdot,\cdot]_{\mathfrak h}=\lambda[\cdot,\cdot]_{\mathfrak g}$. Then $B$ is a Rota--Baxter operator of weight $\lambda$ by Theorem~\ref{thm:differetial_relative_RB} and Example~\ref{ex:lambda_lie_RB}.
\end{proof}

For Rota--Baxter operators of weight $0$, we can consider the following construction. Let $G$ be a Lie group such that the exponential map $\exp\colon\mathfrak g\to G$ is a bijection. Denote $\exp^{-1}=\ln$. Consider the commutative group $H=(G,\ast)$ with the multiplication
\[
g\ast h=\exp(\ln g+ \ln h).
\]
The group $H$ is isomorphic to the additive group $(\mathfrak g,+)$ and the isomorphism is compatible with the adjoint actions $Ad$ of the group $G$ on $H$ and $\mathfrak g$.


\begin{proposition}
    Let $\mathcal B$ be a relative Rota--Baxter operator on $H$. Then the differential $B={\mathcal B}_{*1_G}$ is a Rota--Baxter operator of weight $0$ on the Lie algebra $\mathfrak g$.
\end{proposition}

\begin{proof}
    The Lie algebra $\mathfrak h$ of the group $H$ coincides with $\mathfrak g$ as a linear space and has trivial commutator because $H$ is commutative. By Theorem~\ref{thm:differetial_relative_RB} the operator $B\colon\mathfrak h=\mathfrak g\to\mathfrak g$ is a relative  Rota-Baxter operator on $g$, i.e.
\[
[B(u),B(v)]_{\mathfrak g}= B(\phi(B(u))(v)-\phi(B(v))(u)+[u,v]_{\mathfrak h})=B([B(u),v]_{\mathfrak g}+[u,B(v)]_{\mathfrak g})
\]
for all $u,v\in\mathfrak g$ since $\phi=ad$ is the adjoint operator. Thus, $B$ is a Rota--Baxter operator of weight $0$ on the Lie algebra $g$.
\end{proof}

\bigskip


\section{Some 2-generated Hopf algebras}  \label{CostHopf}

There are some generalizations of the Hopf algebra $H_4$ \cite{A, A1}.  Let $m \geq 2$ be a natural number, $\zeta$ be a primitive $m$-th root from unity in the field $\Bbbk$. It is known that this root exists if and only if $char(\Bbbk) \not|~ m$. An algebra $H_{m^2(\zeta)}$ with unit 1 is generated by elements $g$ and $x$, which satisfy the relations
$$
g^m = 1,~~x^m = 0,~~g x = \zeta x g.
$$
Then elements $(g^i x^k)_{0 \leq i,k \leq m-1}$ form a basis of $H_{m^2(\zeta)}$. Put
$$
\Delta(g) = g \otimes g,~~\Delta(x) = g \otimes x + x \otimes 1,~~\varepsilon(g) = 1,~~\varepsilon(x) = 0,~~S(g) = g^{-1},~~S(x) = -g^{-1} x.
$$
The Hopf algebra $H_{m^2(\zeta)}$ is said to be the Taft algebra. The algebra $H_4 = H_{4(-1)}$ is the  Sweedler algebra.

We suggest another generalization of $H_4$.

\begin{definition}
 Let  $m$ be a non-negative integer, $\zeta \in \Bbbk^*$,  $l \in \mathbb{N} \cup \{+\infty\}$.
If $l \in \mathbb{N}$, then  $f(x) \in \Bbbk[x]$, $deg(f(x)) < l$;  if $l = +\infty$, then $f(x) = 0$. We will denote  $H_{m, \zeta, l, f(x)}$,  an associative  algebra with unit,  that is generated by $g$ and $x$ and satisfies the next relations
$$
g^m = 1, ~~x^l = f(x),~~x g = \zeta g x.
$$
\end{definition}

It is easy to see that any element of $H_{m, \zeta, l, f(x)}$ can be presented as a linear combination of elements
$$
g^{\alpha} x^{\beta},~~0 \leq \alpha \leq m-1,~~0 \leq \beta \leq l-1,
$$
if $l \in \mathbb{N}$ and of elements
$$
g^{\alpha} x^{\beta},~~0 \leq \alpha \leq m-1,~~0 \leq \beta,
$$
if $l = + \infty$.

This algebra generalises some known 2-generated Hopf algebras (see \cite{A}).

\begin{example}
\begin{enumerate}
\item Algebra $A_{\infty}$   over field of characteristic 0 is generated by $g$ and $x$ and is defined by relations
$$
g^2 = 1,~~x g = -g x.
$$
Then $A_{\infty} = H_{2, -1, \infty, 0}$.

\item Algebra $A_{4p}$ over field of characteristic $p \geq 3$ is a $4p$-dimension algebra with generators $g$ and $x$ and with defined  relations
$$
g^2 = 1,~~x g = -g x,~~x^{2p} = 0.
$$
Then $A_{4p} = H_{2, -1, 2p, 0}$.

\item Algebra $A_{(q)}$, $q \in \Bbbk^*$  is a $4p$-dimension algebra with generators $g$ and $x$ and with defined  relations
$$
g^2 = 1,~~x g = -g x,~~x^{2p} = q^{p-1} x^2,~~p \geq 3.
$$
Then $A_{(q)} = H_{2, -1, 2p, q^{p-1} x^2}$.

\end{enumerate}
\end{example}

\begin{theorem} Let $\zeta^m =1$  and $f(x)$ satisfies
$$
f(\zeta x) = \zeta^l f(x).
$$
Then $H_{m, \zeta, l, f(x)}$ is an associative algebra with a linear basis
$$
g^{\alpha} x^{\beta},~~0 \leq \alpha \leq m-1,~~0 \leq \beta \leq l-1,
$$
if $l \in \mathbb{N}$ and with a linear basis
$$
g^{\alpha} x^{\beta},~~0 \leq \alpha \leq m-1,~~0 \leq \beta,
$$
if $l = + \infty$.
\end{theorem}

\begin{proof} For simplicity, let us denote $H = H_{m, \zeta, l, f(x)}$.
Consider an associative algebra $A$ with generators $g$, $x$ and with one relation $xg=\zeta gx$. Then $A$ has linear basis $g^\alpha x^\beta$, $\alpha,\beta\ge 0$.

Let  $J$ be the linear subspace  in $A$ spanned by elements $(g^m-1)g^\alpha x^\beta$ and $g^\alpha x^\beta(x^l-f(x))$ (if $l\in\mathbb N$). The quotient space $\tilde H=A/J$ has linear basis
$$
g^{\alpha} x^{\beta},~~0 \leq \alpha \leq m-1,~~0 \leq \beta \leq l-1,
$$
if $l\in\mathbb N$, and $g^{\alpha} x^{\beta}$,$0 \leq \alpha \leq m-1$, $0 \leq \beta$, if $l=+\infty$.

Let us show that $J$ is an ideal in $A$ and hence, $\tilde H=H$.
Indeed, the equality $\zeta^m = 1$ implies
$$
g^{\alpha'} x^{\beta'} (g^m-1)g^\alpha x^\beta x =\zeta^{\alpha\beta'} (\zeta^{m\beta'}g^m-1)g^{\alpha+\alpha'} x^{\beta+\beta'} = \zeta^{\alpha\beta'}(g^m-1)g^{\alpha+\alpha'} x^{\beta+\beta'}\in J
$$
for any $\alpha', \beta'$.
Further, from the equality $\zeta^l f(x) = f(\zeta x)$ it follows $(x^l-f(x))g =\zeta^l  g(x^l-f(x))$. Hence,
$$
g^\alpha x^\beta(x^l-f(x))g^{\alpha'}x^{\beta'} =\zeta^{\alpha'\beta+\alpha'l}  g^{\alpha+\alpha'}x^{\beta+\beta'}(x^l-f(x))\in J
$$
for any $\alpha', \beta'$.

Hence, $H=\tilde H$ is an algebra with linear basis $g^p x^q$, $0 \leq p < m$, $0 \leq q < l$.
\end{proof}

Further we define  a Hopf algebra structure on $H_{m, \zeta, l, f(x)}$.

Recall that that the quantum binomial coefficients are defined as the coefficients in the equality
$$
(u + v)^p = \sum_{q=0}^p {p \choose q}_{\zeta} u^{p-q} v^q
$$
in the algebra $A_\zeta=\Bbbk[u,v]/(vu-\zeta uv)$.

\begin{theorem} Denote $f(x)=\sum_{p=0}^{l-1}a_px^p$. Let the next conditions hold

\begin{enumerate}
    \item $f(0)=0$, i.e. $a_0=0$.
    \item For any $p$ we have $a_p (g^{l-p} -1) = 0$ $\Leftrightarrow$ for any $p$ such that $a_p \not= 0$ holds $m | (l-p)$;

    \item For any $1 < q < l$ holds ${l \choose q}_{\zeta}= 0$;

    \item For any $p$ such that $a_p \not= 0$ and for any $q$, $1 < q < p$, holds ${p \choose q}_{\zeta}= 0$.

\end{enumerate}

The next operations
$$
\Delta(g) = g \otimes g,~~\Delta(x) = x \otimes 1 + g \otimes x,~~\varepsilon(g) = 1,~~\varepsilon(x) = 0,~~S(g) = g^{-1},~~S(x) = - g^{-1} x
$$
define a Hopf algebra on  $H_{m, \zeta, l, f(x)}$.
\end{theorem}

\begin{proof} For simplicity, let us denote $H = H_{m, \zeta, l, f(x)}$.
Let
$$
f(x) = \sum_p a_p x^p,~~a_p \in \Bbbk.
$$
From the equality $\zeta^l f(x) = f(\zeta x)$ follows that $\zeta^l a_p x^p = a_p \zeta^p x^p$. Hence, for any $p$ such that $a_p \not= 0$, we get $\zeta^{l-p} = 1$.
It equivalent to condition that if $d$ is a minimal number such that $\zeta^d = 1$ is a primitive root, then $d | m$ and $d | (l-p)$ for any $p$ such that $a_p \not=0$.

Let us show that $H$ is a Hopf algebra.

We have $\varepsilon (g^m) = \varepsilon(1) = 1$.

Next relation,
$$
\varepsilon(x g) = \varepsilon(x) \varepsilon (g) = 0 = \varepsilon (\zeta gx) = \zeta \varepsilon (g) \varepsilon (x).
$$

Further, $0 = \varepsilon (x^l) = \varepsilon (f(x)) = f(0) \cdot 1$. Hence, we get condition $f(0) = 0$.

Considering the actions of $\Delta$ on the relations, we get
$$
\Delta(g^m) = \Delta(g)^m = (g \otimes g)^m = g^m \otimes g^m = 1 = \Delta(1).
$$

Further,
\begin{multline*}
\Delta(x g) = \Delta(x) \Delta(g)= \\ = (x \otimes 1+ g \otimes x) (g \otimes g) = xg \otimes g + g^2 \otimes xg = \zeta (g x \otimes g + g^2 \otimes g x) =
\\
= \zeta  (g \otimes g) (x \otimes 1+ g \otimes x) = \zeta \Delta(g) \Delta(x) = \Delta(\zeta g x).
\end{multline*}

The last relation
$$
0 = \Delta(x^l - f(x)) = \Delta(x)^l - f(\Delta(x) )  = (x \otimes 1+ g \otimes x)^l - \sum_p a_p  (x \otimes 1+ g \otimes x)^p =
$$
$$
= \sum_{q=1}^l {l \choose q}_{\zeta} g^{l-q} x^q \otimes x^{l-q} -
\sum_{p<l} a_p \sum_{q=0}^p {p \choose q}_{\zeta} g^{p-q} x^q \otimes x^{p-q} =
$$
$$
= {l \choose 0}_{\zeta} g^{l}  \otimes x^{l} + {l \choose l}_{\zeta} x^{l}  \otimes 1 +
\sum_{q=1}^{l-1}  {l \choose q}_{\zeta} g^{l-q} x^q \otimes x^{l-q} -
\sum_{p<l} a_p \sum_{q=0}^p {p \choose q}_{\zeta} g^{p-q} x^q \otimes x^{p-q} =
$$
$$
= \sum_{p<q} a_p (g^l - g^p) \otimes x^{p} + \sum_{p<q} a_p (1 - 1) x^p \otimes 1 +
$$
$$
+\sum_{q=1}^{l-1}  {l \choose q}_{\zeta} g^{l-q} \, x^{q} \otimes x^{l-q} -
\sum_{p<l} a_p \sum_{q=1}^{p-1} {p \choose q}_{\zeta} g^{p-q} x^q \otimes x^{p-q}.
$$
To have equality $0 = \Delta(x^l - f(x))$ the conditions 1)-3) of proposition must be hold.

We used the following fact
$$
(x \otimes 1) (g \otimes x) = \zeta (g \otimes x)(x \otimes 1),
$$
 hence, 
$$
(x \otimes 1 + g \otimes x)^p = \sum_{q=0}^p {p \choose q}_{\zeta}  gx^{q} \otimes  x^{p-q}.
$$

Let us define a map $S$ by  the formula
$$
S(g^p x^q) = (-1)^q  \zeta^{-pq - q(q-1)/2} g^{-p} x^q,
 $$
and check that $S$ is an antipode.

Suppose that $q=0$, then $S(g^p) = g^{-p}$ and
$$
m (S \otimes \id) \Delta(g^p) = g^{-p} g^p = 1 = \varepsilon(g^p) \cdot 1,~~m (\id \otimes S) \Delta(g^p) = g^{p} g^{-p} = 1 = \varepsilon(g^p) \cdot 1.
$$

Suppose that $q>0$, then
$$
m (S \otimes \id) \Delta(g^p x^q) = (S \otimes \id) (g^p \otimes g^p) (x \otimes 1 + g \otimes x)^q = m (S \otimes \id) \sum_t  {q \choose t}_{\zeta} g^{p+q-t} x^t \otimes g^p x^{q-t} =
$$
$$
=m\left( \sum_{t=0}^q  {q \choose t}_{\zeta} (-1)^t \zeta^{-(p+q-t)t -t(t-1)/2} g^{-(p+q -t)t} x^t \otimes g^p x^{q-t}\right) =
$$
$$
= \sum_{t=0}^q  {q \choose t}_{\zeta} (-1)^t \zeta^{-(p+q-t)t -t(t-1)/2} g^{-(p+q)} \zeta^{pt}  g^p x^{q} =
\sum_{t=0}^q  {q \choose t}_{\zeta} (-1)^t \zeta^{-(q-t)t -t(t-1)/2}   g^{-q} x^{q} =
$$
$$
= \sum_{t=0}^q  {q \choose t}_{\zeta}  \zeta^{t(t-1)/2}  (- \zeta^{1-q})^t g^{-q} x^{q} = \prod_{t=0}^{q-1} (1 + \zeta^t (- \zeta^{-(q-1)}))\cdot g^{-q} x^{q}= 0 = \varepsilon(g^p x^q).
$$
Here we used the Cauchy binomial theorem:
$$
\prod_{t=0}^{q-1} (1 + \zeta^t u) = \sum_{t=0}^{q} {q \choose t}_{\zeta} \zeta^{t(t-1)/2} u^t.
$$

On the other side,
$$
m(\id \otimes S) \Delta(g^p x^q) = m(\id \otimes S) \sum_{t=0}^q  {q \choose t}_{\zeta}     g^{p+q-t} x^{qt} \otimes g^p x^{q-t} =
$$
$$
=\sum_{t=0}^q  {q \choose t}_{\zeta}  g^{p+q-t} x^{t}  (-1)^{q-t} \zeta^{-p(q-t) - (q-t)(q-t-1)/2} g^{-p-q+t} x^{q-t} =
$$
$$
=\sum_{t=0}^q  {q \choose t}_{\zeta}   (-1)^{q-t} \zeta^{-pq - t(q-t) - (q-t)(q-t-1)/2}x^{q} =
$$
$$
= \zeta^{-pq} \sum_{s=0}^q  {q \choose s}_{\zeta} (-1)^s \zeta^{-s(q-s) -  s(s-1)/2} x^{q}
= \zeta^{-pq} \sum_{s=0}^q  {q \choose s}_{\zeta}  \zeta^{s(s-1)/2} (-\zeta^{q+1})^s \cdot x^{q}= 0,
$$
where $s = q-t$ and we used the property of binomial coefficients,
$$
 {q \choose t}_{\zeta}  =  {q \choose q-t}_{\zeta}.
$$
\end{proof}

\begin{theorem}
1) If $m \not=0$, then $\psi \in \Aut(H_{m, \zeta, l, f(x)})$ has the form
$$
\psi(g) = g^k, ~~\psi(x) =  \sum_{q\equiv k \pmod m} c_{0,q} x^q,
$$
where
\begin{enumerate}
\item  $(k,m) = 1$,

\item for any $q$ such that $c_{0,q} \not= 0$ and for any $0 < t < q$ we have ${q \choose t}_{\zeta} = 0$,

\item $k^2 \equiv 1 \pmod d$,

\item $(u^l - f(u)) | (\psi(u)^l - f(\psi(u)))$ in $\Bbbk[u]$.

\end{enumerate}

2) If $m =0$, then $\psi \in \Aut(H_{m, \zeta, l, f(x)})$ has the form $\psi(g) = g$,  $\psi(x) =  c x$, $c \in \Bbbk$, where
 for any $p$ such that $a_p \not= 0$  we have $c^{l-p} = 1$.

 In particular, if $f(x) \not= 0$, then $\Aut(H_{m, \zeta, l, f(x)}) \cong \mathbb{Z}_N$, where $N = \gcd \{l-p ~|~a_p \not= 0 \}$.
  If $f(x) = 0$, then $\Aut(H_{m, \zeta, l, f(x)}) \cong \Bbbk^*$.

\end{theorem}

\begin{proof}
Let $\psi \in \Aut(H_{m, \zeta, l, f(x)})$. Then image of $g$ is a group-like element. Hence, $\psi(g) = g^k$, where $(k, m) = 1$ if  $m \not=0$, and $\psi(g) = g^{\pm 1}$  if  $m=0$.

Let us find the image of $x$. Suppose that
$$
\psi(x) = \sum_{p=0}^{m-1} \sum_{q=0}^{l-1} c_{p,g} g^p x^q,~~c_{p,g} \in \Bbbk.
$$
We have to prove that $\Delta \psi(x) = \psi (\Delta(x))$. We have  $\psi (\Delta(x)) = \psi(x) \otimes 1 + g^k \otimes \psi(x)$ and
$$
\Delta \psi(x) = \sum_{p=0}^{m-1} \sum_{q=0}^{l-1} c_{p,q} \sum_{t=0}^q {q \choose t}_{\zeta} g^{p+q-t} x^t \otimes g^p x^{q-t}.
$$
Hence, one can present $\psi(x)$ in the form
$$
\psi(x) = \sum_{q\equiv k \pmod m} c_{0,q} x^q.
$$
We have condition, that for any $q$ such that $c_{0,q} \not= 0$ and for any $0 < t < q$ we have ${q \choose t}_{\zeta} = 0$.

Let us check the relation $\psi(x g) = \zeta \psi(g x)$. The left hand side has the form
$$
\psi(x g) = \sum_q c_{0,q} x^q g^k = g^k \sum_q c_{0,q} \zeta^{kq} x^q.
$$
The  right hand side has the form
$$
\psi(\zeta g x) =  g^k \sum_q c_{0,q} \zeta x^q.
$$
Hence, for any $q$ such that $ c_{0,q} \not= 0$ we have $\zeta^{kq} = \zeta$. From it follows that $kq \equiv 1 \pmod d$. Take in attention that $q \equiv k \pmod d$, we get $k^2 \equiv 1 \pmod d$.

From relation $0 = \psi(x^l -f(x)) = \psi(x)^l -f(\psi(x))$ follows that $u^l - f(u)$ divides $\psi(u)^l - f(\psi(u))$ in $\Bbbk[u]$.
\end{proof}

\begin{corollary}
1) $\Aut_{Hopf}(H_4) = \Bbbk^*$;

2) $\Aut_{Hopf}(A_{\infty}) = \Bbbk^*$.
\end{corollary}

\bigskip


\section*{Acknowledgments}
Authors are grateful to V.~N.~Zhelyabin, V.~Gubarev, M.~Goncharov, A. Pozhidaev,  and P. Kolesnikov for the fruitful discussions and useful suggestions.
Authors are also grateful to participants of the seminar ``\'{E}variste Galois''
at Novosibirsk State University for attention to our work.

Valery G. Bardakov is supported by the state contract of the Sobolev Institute of Mathematics, SB RAS (no. I.1.5, project FWNF-2022-0009).


\medskip

\noindent Valeriy G. Bardakov \\
Sobolev Institute of Mathematics,  Acad. Koptyug ave. 4, 630090 Novosibirsk, Russia; \\
Novosibirsk State Agrarian University,
Dobrolyubova str., 160, 630039 Novosibirsk; \\
Regional Scientific and Educational Mathematical Center of Tomsk State University; \\
Lenin ave. 36, 634009 Tomsk, Russia \\
email: bardakov@math.nsc.ru


\medskip
\noindent Igor Nikonov \\
Lomonosov Moskow  State University \\
email: nikonov@mech.math.msu.su

\end{document}